\documentclass[reqno,twoside,12pt]{amsart}

\hoffset - 1.5cm

\usepackage{amsmath}
\usepackage{amssymb}
\usepackage{graphicx}
\usepackage{amstext}
\usepackage{amsopn}
\usepackage{amsfonts}
\usepackage{xcolor}
\usepackage{setspace}
\usepackage{enumerate}
\usepackage[polish,english]{babel}
\usepackage[utf8]{inputenc}
\usepackage{polski}

\newtheorem{Remark}{Remark}[section]
\newtheorem{Corollary}[Remark]{Corollary}
\newtheorem{Definition}[Remark]{Definition}

\newtheorem{Fact}[Remark]{Fact}
\newtheorem{Lemma}[Remark]{Lemma}

\newtheorem{Theorem}[Remark]{Theorem}

\newcommand{\bH}{\mathbb{H}}
\newcommand{\bI}{\mathbb{I}}

\newcommand{\bN}{\mathbb{N}}

\newcommand{\bR}{\mathbb{R}}

\newcommand{\bU}{\mathbb{U}}

\newcommand{\bV}{\mathbb{V}}

\newcommand{\cA}{\mathcal{A}}

\newcommand{\cC}{\mathcal{C}}

\newcommand{\cG}{\mathcal{G}}
\newcommand{\cH}{\mathcal{H}}
\newcommand{\cI}{\mathcal{I}}

\newcommand{\cN}{\mathcal{N}}
\newcommand{\cO}{\mathcal{O}}

\newcommand{\cT}{\mathcal{T}}
\newcommand{\cU}{\mathcal{U}}
\newcommand{\cV}{\mathcal{V}}
\newcommand{\cW}{\mathcal{W}}

\numberwithin{equation}{section} \errorcontextlines=0

\newcommand{\im}{\mathrm{ im \;}}

\newcommand{\sone}{SO(2)}

\newcommand{\ds}{\displaystyle}
\newcommand{\nt}{\noindent}

\newcommand{\h}{\mathbb{H}}
\newcommand{\sub}{\overline{\operatorname{sub}}}
\newcommand{\G}{\mathcal{G}}

\begin{document}

\title[Bifurcations from the orbit]{Bifurcations from the orbit of solutions of the Neumann problem}

\author{Anna Go\l\c{e}biewska}
\address{Faculty of Mathematics and Computer Science\\
Nicolaus Copernicus University \\
PL-87-100 Toru\'{n} \\ ul. Chopina $12 \slash 18$ \\
Poland, ORCID 0000--0002--2417--9960}

\email{Anna.Golebiewska@mat.umk.pl}

\author{Joanna Kluczenko}
\address{Faculty of Mathematics and Computer Science, University of Warmia and Mazury \\ul.~Sloneczna 54, PL-10-710 Olsztyn, Poland,
ORCID 0000--0003--0121--1136}

\email{jgawrycka@matman.uwm.edu.pl}

\author{Piotr Stefaniak}
\address{School of Mathematics,
West Pomeranian University of Technology \\PL-70-310 Szcze\-cin, al. Piast\'{o}w $48\slash 49$, Poland, ORCID 0000--0002--6117--2573}

\email{pstefaniak@zut.edu.pl}

\numberwithin{equation}{section}
\allowdisplaybreaks
\date{\today}

\maketitle

\begin{abstract}
The purpose of this paper is to study weak solutions of a nonlinear Neumann problem considered on a ball. Assuming that the potential is invariant, we consider an orbit of critical points, i.e. we do not assume that critical points are isolated. We apply techniques of the equivariant analysis to examine bifurcations from the orbits of trivial solutions.
We formulate sufficient conditions for local and global bifurcations, in terms of the right-hand side of the system and eigenvalues of the Laplace operator.
Moreover, we characterise orbits at which the global symmetry-breaking phenomenon occurs.
\end{abstract}

\section{Introduction}

In this paper, we study bifurcations of weak solutions of elliptic systems of the form:
\begin{equation} \label{eq:neumannwstep1}
\left\{
\begin{array}{rclcl}   - \triangle u & =& \lambda \nabla F(u )   & \text{ in   } & B^N \\
                   \frac{\partial u}{\partial \nu} & =  & 0 & \text{ on    } & S^{N-1},
\end{array}\right.
\end{equation}

\nt where $B^N$ is the open unit ball in $\bR^N$, $S^{N-1}=\partial B^N$ and the function $F\colon \bR^m \to \bR$ satisfies additional assumptions, see Section 2.

In particular, we are interested in the equivariant case. Namely, we assume that on the space $\bR^m$ there is defined an action of the compact Lie group $\Gamma$ and $\nabla F$ is the $\Gamma$-equivariant mapping. Moreover, it is known that $B^N$ is $SO(N)$-invariant, where $SO(N)$ stands for the special orthogonal group in dimension $N$.

Consider the set $\nabla F^{-1}(0).$ For $u_0 \in \nabla F^{-1}(0)$ the constant function $\tilde{u}_0 \equiv u_0$ is a solution of \eqref{eq:neumannwstep1} for all $\lambda \in \bR$. Therefore, we obtain the family of trivial solutions $\{\tilde{u}_0\} \times \bR$. Investigating the change of the Conley index for different levels $\lambda \in \mathbb{R}$, one can obtain a sequence of nontrivial weak solutions bifurcating from the point $(\tilde{u}_0, \lambda_0)$, for some values $\lambda_0 \in \mathbb{R}.$ Investigating the change of the topological degree, one can prove the existence of the continuum, containing $(0, \lambda_0),$ of nontrivial weak solutions of the system (i.e. the global bifurcation of weak solutions).

For a system of elliptic differential equations with Dirichlet boundary conditions such methods have been used in many papers, among others by the first and the second author  in  \cite{GawRyb},  \cite{GolRyb1}, \cite{Klu}.  A similar method has been also used   in \cite{GolKlu} for the system with the Neumann boundary conditions with the infinity instead of the critical point.
The phenomenon of symmetry breaking for elliptic systems with the Neumann boundary conditions has been considered by the third author in \cite{Ste}.

The results described above are obtained with the assumption that $u_0$ is an isolated critical point of the potential $F$.

Assuming that $\nabla F$ is a $\Gamma$-equivariant mapping, we obtain that for $u_0 \in \nabla F^{-1}(0)$ also $\gamma u_0 \in \nabla F^{-1}(0)$ for all $\gamma \in \Gamma$. It is therefore clear, that the assumption that the critical point $u_0$ is an isolated one, does not have to be satisfied in this case.

The method, that can be used in this situation, is an investigating of the index of the isolated orbit. Under some additional assumptions, this method has been recently proposed by Perez-Chavela, Rybicki and Strzelecki in \cite{PRS}. In this paper it has been proved that the computation of the Conley index of the orbit can be in some cases reduced to computation of the index of a point from the space normal to the orbit.

To study weak solutions of the system \eqref{eq:neumannwstep1} we apply variational methods, i.e. we associate with the system a functional $\Phi$ defined on a suitable Hilbert space $\bH$. Its critical points are in one-to-one correspondence with weak solutions of \eqref{eq:neumannwstep1}. The tools we use are the finite and infinite dimensional equivariant Conley index (see \cite{[Bartsch]}, \cite{Geba} for the definition in the finite dimensional case and \cite{Izydorek} for the infinite dimensional case) and the degree for invariant strongly indefinite functionals, defined in \cite{GolRyb1}.

Consider the group $\cG=\Gamma \times SO(N).$ Since $\bR^m$ is a $\Gamma$-representation and $B^N$ is an $SO(N)$-invariant set, the space  $\bH$ is a  $\cG$-representation. It occurs that for $u_0 \in (\nabla F)^{-1}(0)$, $(g\tilde{u}_0, \lambda)$ is a critical point of $\Phi$ for all $g\in \cG, \lambda \in \bR.$

 Therefore we can consider the set of trivial solutions $\cT=\cG(\tilde{u}_0) \times \bR$. We are going to investigate bifurcations of nontrivial solutions from the family $\cT$. Our aim is to formulate necessary and sufficient conditions, in terms of the right-hand side of the system and of the eigenvalues of the Laplace operator, for a bifurcation from the orbit $\cG(\tilde{u}_0) \times\{\lambda_0\}$.

 We also consider the global symmetry-breaking phenomenon at the orbit $\cG(\tilde{u}_0) \times \{\lambda_0\}$.
More precisely, knowing that the trivial solutions are radial, we study when the bifurcating solutions are non-radial. The analogous problem has been studied by the third author in \cite{RybShiSte} and \cite{RybSte} on the sphere and on the geodesic ball, with the use of the lemma due to Dancer (see \cite{Dancer1979}), characterising isotropy groups of bifurcating solutions. In our situation, if the group $\Gamma$ is not a discrete one, we cannot use this result. Therefore we generalise it.

After this introduction the paper is organised in the following way:

In Section \ref{sec:preliminaries} we introduce the problem and recall some definitions.  With an elliptic system on a ball we associate a functional. Next we study the properties of the linear system. We end this section with the definitions of local and global bifurcations from an orbit and of the admissible pair.

In Section 3 we formulate and prove the main results of this article, namely Theorems \ref{th:BIF} and \ref{th:GLOB} concerning the local and global bifurcations of solutions, and Theorem \ref{thm:SymmBreak}, concerning the symmetry breaking problem. First we consider the phenomenon of bifurcation from the critical orbit. We start with some auxiliary results. In Lemma \ref{fact:warunekkonieczny} we describe the set of parameters at which the bifurcation of solutions can occur. In Theorem \ref{th:zmiana_indeksu} we investigate the change of the Conley index at the levels obtained in Lemma \ref{fact:warunekkonieczny}. This result is applied to prove Theorems \ref{th:BIF} and \ref{th:GLOB}. The local bifurcation of solutions, under weaker assumptions, is considered also in Theorem \ref{thm:0}. Next we study the symmetry breaking problem. In Theorem \ref{thm:SymmBreak} we prove the bifurcation of orbits of non-radial solutions emanating from orbits of radial ones. To obtain this result, we generalise the result of Dancer in Lemma \ref{lem:IsGr}.

In Section \ref{sec:illustration} we illustrate our results with a few examples. Using the properties of the eigenspaces of the Laplace operator (with the Neumann boundary conditions) on the ball, we verify assumptions of our main results.

Section \ref{sec:app} is the appendix. In the main part of our paper we assume that the reader is familiar with some classical definitions and facts, concerning for example the equivariant Conley index or the properties of eigenspaces of the Laplace operator on a ball. However, it is not easy to
find the full description of these properties. Therefore, for the completeness of the paper we collect in this section the information which we use to prove our main results. In this section we present also an equivariant version of the implicit function theorem in infinite dimensional spaces, due to Dancer.

%===================================================

\subsection{Notation}
Suppose that $G$ is a compact Lie group. We denote by $\sub(G)$ the set of closed subgroups of $G$. For $u$ from a given $G$-space $X$ we denote by the $G(u)$  the orbit through $u$ and $G_u$ stands for the isotropy group of $u$.

Further, by $U(G)$ we denote the Euler ring of $G$ and we use the symbol $\chi_G(\cdot)$ to denote the $G$-equivariant Euler characteristic of a pointed finite $G$-CW-complex. Moreover, the symbols $CI_{G}(S, f)$ and $\mathcal{CI}_{\cG}(S, f)$ stand for the Conley indices of an isolated invariant set $S$ of the flow generated by $f$, considered respectively in finite and infinite dimensional cases. More precise description can be found in Appendix.

Finally, for a Hilbert space $\h$ and  $u_0 \in \h$ we denote by $B_{\delta}(u_0,\h)$ (respectively $D_{\delta}(u_0,\bH)$) the open (respectively closed) ball in $\bH$ centred at $u_0$ and with radius $\delta$. In particular,  we use the symbol $B^N$ for the open ball if $\delta=1$, $u_0=0$ and $\h=\bR^N$ and we write $S^{N-1}$ for $\partial B^N.$

%===================================================
\section{Preliminaries}\label{sec:preliminaries}

Throughout this paper $\Gamma$ stands for a compact Lie group and $\bR^m$ is an orthogonal representation of the group $\Gamma$. Consider $F\colon \bR^m \to \bR$ satisfying:
\begin{enumerate}

\item[(B1)] $F \in C^2(\bR^m , \bR)$ is such that for every $u \in \bR^m$ we have $|\nabla^2 F(u) | \leq a + b |u|^{q}$ where $a,b \in \bR$ and $1<q < \frac{4}{N-2}$ for $N \geq 3$ and $ 1<  q< \infty $ for $N=2,$
\item[(B2)] $F$ is $\Gamma$-invariant, i.e. $F(\gamma u) =F(u)$ for every $\gamma \in \Gamma$, $u\in\bR^m$.

\end{enumerate}
Our aim is to study bifurcations of weak solutions of the nonlinear Neumann problem, parameterised by $\lambda \in \bR$,
\begin{equation} \label{eq:neumann}
\left\{
\begin{array}{rclcl}   - \triangle u & =& \lambda \nabla F(u )   & \text{ in   } & B^N \\
                   \frac{\partial u}{\partial \nu} & =  & 0 & \text{ on    } & S^{N-1}.
\end{array}\right.
\end{equation}

\nt Denote by $H^1(B^N)$ the first Sobolev space on $B^N$ and consider a separable Hilbert space $\bH=
\bigoplus_{i=1}^{m} H^1(B^N)$ with the scalar product
\begin{equation}\label{iloczyn} \ds \langle v, w \rangle_{\bH} = \ds \sum_{i=1}^m \langle v_i, w_i
\rangle_{H^1( B^N)} = \sum_{i=1}^m  \int\limits_{B^N} (\nabla v_i(x), \nabla w_i(x)) + v_i(x)\cdot w_i(x) dx. \end{equation}

Denote by $\G$ the group $\Gamma \times SO(N)$, where $SO(N)$ is the special orthogonal group in dimension $N$. Note  that the space $\h$ with the scalar product given by \eqref{iloczyn} is an orthogonal ${\G}$-representation with the ${\G}$-action given by
\begin{equation}\label{eq:action}
(\gamma, \alpha) (u)(x)= \gamma u({\alpha}^{-1}x)\  \text{ for }\  (\gamma, \alpha) \in {\G}, u \in \h, x\in B^N.
\end{equation}

 It is well known  that weak solutions of the problem \eqref{eq:neumann} are in one-to-one correspondence with critical points (with respect to $u$) of the functional $\Phi \colon \h \times \bR \rightarrow \bR$
defined by
\begin{equation}\label{eq:Phi}
\Phi (u,\lambda) = \frac{1}{2} \int\limits_{B^N} |\nabla u(x)|^2 dx- \lambda\int\limits_{B^N} F(u(x))dx.
\end{equation}
Computing the gradient of $\Phi$ with respect to $u$ we obtain:
\begin{equation}\label{eq:gradientPhi}
\langle \nabla_{u} \Phi (u,\lambda), v \rangle_{\h} =
\int\limits_{B^N} (\nabla u(x), \nabla v(x) ) - (\lambda \nabla F (u(x)), v(x) ) dx, \ u,v \in \h.
\end{equation}
Moreover,
\begin{equation*}
\begin{split}
\left<\nabla^2_u \Phi(u, \lambda)w, v\right>_{\h}=\int\limits_{B^N} ( \nabla w(x),\nabla v(x)) -(\lambda \nabla^2 F(u(x))w(x), v(x)) dx, \ u, w, v \in \h.
\end{split}
\end{equation*}

Assumption (B2) implies that $\nabla_u \Phi \colon\h \times \bR \rightarrow \h$ is ${\G}$-equivariant.

Moreover, from imbedding theorems and the assumption (B1) it follows that the operator $\nabla_u \Phi$ is a completely continuous perturbation of the identity.

\subsection{Linear equation}\label{linear}
In this subsection we consider the equation \eqref{eq:neumann} in the linear case, i.e. the system:
 \begin{equation} \label{eq:lin_neumann}
\left\{
\begin{array}{rclcl}  - \triangle u & = &\lambda A u  & \text{ in   } & B^N \\
                   \frac{\partial u}{\partial \nu} & =  & 0 & \text{ on    } & S^{N-1},
\end{array}
\right.
\end{equation}
where $A$ is a real, symmetric $(m\times m)$-matrix.

Using formula \eqref{eq:Phi} we can associate with \eqref{eq:lin_neumann} the functional   $\Phi_A \colon \h \times \bR \rightarrow \bR$ given by
\begin{equation}\label{eq:phi}
\Phi_A (u,\lambda) = \frac{1}{2} \int\limits_{B^N} |\nabla u(x)|^2 dx - \frac{\lambda}{2} \int\limits_{B^N}  (Au(x),u(x))dx.
\end{equation}
Note that from \eqref{eq:gradientPhi} for every $v\in\h$ we have \begin{equation*}
 \langle \nabla_u \Phi_A(u,\lambda),v\rangle_{\h}=  \langle u, v \rangle_{\h} - \langle L_{\lambda A} u, v \rangle_{\h},
\end{equation*}

\nt where
\begin{equation}\label{LA} \langle L_{\lambda A} u, v \rangle_{\h} = \int\limits_{B^N} (u(x), v(x)) + (\lambda A u(x), v(x)) dx.  \end{equation}
The existence and boundedness of the operator $L_{\lambda A}\colon \h \rightarrow \h$ follow from the Riesz theorem. By definition $L_{\lambda A}$ is self-adjoint.

Let us denote  by $\sigma(-\Delta; B^N) = \{ 0=
\beta_1 < \beta_2 < \ldots < \beta_k < \ldots\}$ the set of distinct
eigenvalues of the Laplace operator (with the Neumann boundary conditions) on the ball. Write $\bV_{-\Delta}(\beta_k)$ for the eigenspace of $-\Delta$ corresponding to  $\beta_k \in
\sigma(-\Delta; B^N)$. In Appendix  we give a more precise  description of these eigenspaces. By the spectral theorem it follows that $\ds H^1(B^N) = cl ( \bigoplus_{k=1}^{\infty} \bV_{-\Delta} (\beta_k)).$
Let us denote by $\h_k$ the space $\ds \bigoplus_{i=1}^m \bV_{-\Delta}(\beta_k).$
In particular, $u=\sum\limits_{k=1}^{\infty}u_k$ for every $u\in \bH$, where $u_k\in\h_k$.

Let $\alpha_1,\ldots,\alpha_m$ denote the eigenvalues of $A$ (not necessarily distinct) with corresponding eigenvectors $f_1,\ldots,f_m$, which form an orthonormal basis of $\bR^m$.

Let $\pi_j\colon\bH\to H^1(B^N)$ be a projection such that $\pi_j(u)(x)=(u(x),f_j)$, $j=1,\ldots, m$.
Clearly, if $u_k \in \bH_k,$ then $\pi_j(u_k) \in \bV_{-\Delta}(\beta_k)$ for $j=1, \ldots, m.$

In the lemma below we characterise the operator $L_{\lambda A}$, given by the formula \eqref{LA}.

\begin{Lemma}\label{operatorLlambdaA}
For every $u\in\bH$
\[
L_{\lambda A} u=\sum\limits_{k=1}^{\infty}\sum\limits_{j=1}^m \frac{1+\lambda\alpha_j}{1+\beta_k} \pi_j(u_k) \cdot f_j.
\]
\end{Lemma}
The proof of this lemma is standard, see for example the proof of Lemma 3.2  in \cite{GolKlu}.

Let us denote by $\sigma(L)$ the spectrum of a linear operator $L\colon \h \to \h$.
From the above lemma there  immediately follows the corollary:
\begin{Corollary}\label{spektrumLA}

Let $L_{\lambda A}$ be defined by \eqref{LA}. Then:
\[
\sigma(L_{\lambda A})=\left\{
 \frac{1+\lambda\alpha_j}{1+\beta_k}\colon \alpha_j \in\sigma(A), \beta_k\in \sigma(-\Delta; B^N)
\right\}.
\]
Moreover,
\[
\sigma(Id-L_{\lambda A})=\left\{
 \frac{\beta_k-\lambda\alpha_j}{1+\beta_k}\colon \alpha_j \in\sigma(A), \beta_k\in \sigma(-\Delta; B^N)
\right\}.
\]

\end{Corollary}

Fix eigenvalues $\alpha_{j_0}\in \sigma(A)$ and $\beta_{k_0}\in\sigma(-\Delta; B^N)$. Let $\bV_A(\alpha_{j_0})$ be the eigenspace associated with the eigenvalue $\alpha_{j_0}$ and $\mu_{A}(\alpha_{j_0}) = \dim \bV_A(\alpha_{j_0})$. Let $\Pi_{j_0}\colon \bR^m\to\bR^m$ be an orthogonal projection such that $\Pi_{j_0}(\bR^m)=\bV_A(\alpha_{j_0})$ and define $\tilde{\Pi}_{j_0}\colon \bH\to \bH$ by $(\tilde{\Pi}_{j_0}(u))(x)=\Pi_{j_0}(u(x))$. Denote
\[
\bV_{-\Delta}(\beta_{k_0})^{\mu_{A}(\alpha_{j_0})}=\tilde{\Pi}_{j_0}\left(\bigoplus\limits_{j=1}^m\bV_{-\Delta}(\beta_{k_0}) \right).
\]
It follows that
\[
\bV_{-\Delta}(\beta_{k_0})^{\mu_{A}(\alpha_{j_0})}=
\mathrm{span}\left\{
h\cdot f\colon h\in \bV_{-\Delta}(\beta_{k_0}), f\in \bV_A(\alpha_{j_0})
\right\}\subset\bH.
\]

From Lemma \ref{operatorLlambdaA} we obtain:

\begin{Corollary}\label{cor:kernel}
 If $\sigma(\lambda A)\cap \sigma(-\Delta; B^N)=\{\alpha_{j_1},\ldots, \alpha_{j_s}\}$, then
\[
\ker( Id - L_{\lambda A})=\bV_{-\Delta}(\alpha_{j_1})^{\mu_{\lambda A}(\alpha_{j_1})}\oplus\ldots\oplus
\bV_{-\Delta}(\alpha_{j_s})^{\mu_{\lambda A}(\alpha_{j_s})}.
\]

\end{Corollary}

\subsection{Notion of the bifurcation from the critical orbit.}\label{subsec:bifurcation}
Fix $u_0 \in (\nabla F)^{-1}(0)$. Since $F$ is $\Gamma$-invariant, and therefore $\nabla F$ is $\Gamma$-equivariant, $\gamma u_0 \in (\nabla F)^{-1}(0)$ for all $\gamma \in \Gamma$, i.e. $\Gamma(u_0)  \subset (\nabla F)^{-1}(0).$ We  call such a set a critical orbit of $F$.

Note that $T_{u_0} \Gamma (u_0) \subset \ker \nabla^2 F (u_0)$ and therefore $\dim \ker \nabla^2 F (u_0) \geq \dim T_{u_0} \Gamma (u_0) = \dim \Gamma (u_0) .$ We  assume that in this inequality there holds:  \begin{equation}\label{eq:orbita} \dim \ker \nabla^2 F (u_0) = \dim \Gamma (u_0).\end{equation}
We  call such an orbit a non-degenerate one.

By the equivariant Morse lemma, see \cite{[Wass]}, from  \eqref{eq:orbita} we conclude that $\Gamma (u_0)$ is isolated in $(\nabla F)^{-1}(0).$

Since $u_0 \in (\nabla F)^{-1}(0)$, a constant function $\tilde{u}_0\equiv u_0$ is a solution of the problem \eqref{eq:neumann} for all $\lambda \in \mathbb{R}$. Therefore,  $(\tilde{u}_0, \lambda)$, and consequently $(\gamma\tilde{u}_0, \lambda)$ for every $\gamma\in \Gamma$, is a critical point of the functional $\Phi$ given by \eqref{eq:Phi}. Since from \eqref{eq:action} we have $\G(\tilde{u}_0)=\Gamma(\tilde{u}_0)$, we obtain a critical orbit of $\Phi$ and therefore a ${\G}$-orbit of weak solutions of \eqref{eq:neumann} for all $\lambda \in \mathbb{R}$. Hence we can
consider a family of solutions $\cT = {\G}(\tilde{u}_0) \times \bR \subset \bH \times \bR$. We call the elements of $\cT$ the trivial solutions of \eqref{eq:neumann}. Put
$\cN=\{(v, \lambda) \in (\bH \times \bR) \setminus \cT\colon \nabla_v\Phi(v, \lambda)=0\}.$

\begin{Definition}
A local bifurcation from the orbit ${\G}(\tilde{u}_0) \times \{\lambda_0\} \subset \cT$ of solutions of \eqref{eq:neumann} occurs if the point $(\tilde{u}_0, \lambda_0)$ is an accumulation point of the set $\cN$.
\end{Definition}

\begin{Remark} Note that if $(\tilde{u}_0, \lambda_0)$ is an accumulation point of $\cN$ then for all $g \in \G$, $(g\tilde{u}_0, \lambda_0)$ is also an accumulation point. Therefore $\cG(\tilde{u}_0)\subset cl(\cN)$.
\end{Remark}

\begin{Definition}
A global bifurcation from the orbit ${\G}(\tilde{u}_0) \times \{\lambda_0\} \subset \cT$ of solutions of \eqref{eq:neumann} occurs if there is a connected component $\cC(\lambda_0)$ of $cl (\cN)$ such that either $\cC(\lambda_0)\cap (\cT \setminus ({\G}(\tilde{u}_0)\times \{\lambda_0\}))\neq \emptyset$ or $\cC(\lambda_0)$ is unbounded.
\end{Definition}

The set of all $\lambda_0 \in \bR$ such that a local (respectively global) bifurcation from  the orbit ${\G}(\tilde{u}_0) \times \{\lambda_0\} $ occurs we denote by $BIF$ (respectively $GLOB$).
Note that directly from the above definitions it follows that $GLOB \subset BIF.$

\subsection{Admissible pair}\label{subsec:adm} The notion of the admissible pair has been introduced in \cite{PRS}.
Fix a compact Lie group $G$ and  let $H\in\sub(G)$.
Denote by $(H)_G$ the conjugacy class of~$H.$

\begin{Definition}
 A pair $(G,H)$ is called admissible, if for any $K_1,K_2\in\sub(H)$ the following condition is satisfied: if $(K_1)_H\neq (K_2)_H$, then $(K_1)_G\neq (K_2)_G$.
\end{Definition}

\begin{Lemma}\label{admissible}
The pair $(\Gamma\times SO(N), \{e\}\times SO(N))$ is admissible.
\end{Lemma}
\begin{proof}
Let us denote by H the group $\{e\} \times SO(N)$ and recall that $\G=\Gamma \times SO(N).$
Moreover, let $\tilde{K}_1, \tilde{K}_2 \in \sub(H).$ By definition of $H$ there are  $K_1, K_2 \in \sub(SO(N))$ such that $\tilde{K}_1 = \{e\} \times K_1$ and
$\tilde{K}_2 = \{e\} \times K_2$.
Suppose that
$(\tilde{K}_1)_{\G}= (\tilde{K}_2)_{\G}$, i.e. $(\{e\} \times K_1)_{\G} =(\{e\} \times K_2)_{\G}.$ Therefore there exists ${(\gamma, \alpha )\in {\G}}$ such that $\{e\} \times K_1 = (\gamma, \alpha )(\{e\} \times K_2)(\gamma, \alpha )^{-1}$ and hence
\begin{eqnarray*}
\{e\} \times K_1=\{\gamma e  \gamma^{-1}\} \times \alpha  K_2 \alpha ^{-1} = \{e\} \times \alpha  K_2 \alpha ^{-1}= (e,\alpha ) (\{e\} \times K_2)(e,\alpha )^{-1}.
\end{eqnarray*}
Thus $(\tilde{K}_1)_H = (\tilde{K}_2)_H$ and the proof is complete. \end{proof}

%=====================================================
\section{Main Results}\label{sec:main}

Consider the nonlinear system \eqref{eq:neumann} with a potential $F$ satisfying (B1), (B2). Fix $u_0 \in (\nabla F)^{-1}(0)$ such that the orbit $\Gamma(u_0)$ is non-degenerate. We put two additional assumptions:

\begin{enumerate}
\item[(B3)] $F(u) = \frac{1}{2} ( A (u-u_0), u-u_0 ) + g(u-u_0),$ where
$A $ is a real symmetric $(m\times m)$-matrix and  $\nabla g (u) = o(|u|)$ for $|u| \rightarrow 0$,

\item[(B4)] $\Gamma_{u_0}=\{e\}.$
\end{enumerate}

From the assumption (B3) we conclude that the gradient of the functional associated with the equation
\eqref{eq:neumann} has the following form:
$$\nabla_u \Phi (u, \lambda)= u - \tilde{u}_0 - L_{\lambda A}(u-\tilde{u}_0)+\lambda \nabla \eta(u-\tilde{u}_0),$$
where
$L_{\lambda A}\colon\h \rightarrow \h$ is a $\G$-equivariant operator given by \eqref{LA}. Moreover, $\nabla \eta\colon\h \rightarrow \h$ given by $\langle \nabla \eta(u), v\rangle_{\h}= \int_{B^N}(\nabla g(u(x)),v(x)) dx$ is   a $\G$-equivariant operator such that $\nabla \eta(u)=o(|u|_{\h})$ for $|u|_{\h}\rightarrow 0.$

From the assumption (B4) it follows that $\G_{\tilde{u}_0}=\{e\} \times SO(N).$

\subsection{Bifurcation from the critical orbit.}\label{bifurcations}

Following the standard notation we denote the linear part of $\nabla_u \Phi(\cdot, \lambda)$ at $\tilde{u}_0$ by $\nabla^2_u \Phi(\tilde{u}_0, \lambda),$ thus $\nabla^2_u \Phi(\tilde{u}_0, \lambda)u=u - L_{\lambda A} u$.

Let us denote by $\Lambda$ the set
$\bigcup_{\alpha_j \in \sigma(A) \setminus\{0\}} \bigcup_{\beta_k \in \sigma(-\Delta;B^N)} \{\frac{\beta_k}{\alpha_j}\}.$

\begin{Lemma}\label{fact:warunekkonieczny}
If $\lambda_0 \in BIF,$ then $\lambda_0 \in \Lambda.$
\end{Lemma}

\begin{proof}
We first observe that for all $\lambda \in \bR$, since $\G(\tilde{u}_0)$ is a critical orbit of $\Phi(\cdot,\lambda)$, we have $\dim \ker \nabla^2_u\Phi(\tilde{u}_0,\lambda)\geq \dim (\G(\tilde{u}_0)\times\{\lambda\})$.

Moreover if $\lambda_0 \in BIF$, this inequality is a strict one. Indeed, if $\dim \ker \nabla^2_u\Phi(\tilde{u}_0,\lambda_0)= \dim (\G(\tilde{u}_0)\times\{\lambda_0\})$, then by the equivariant implicit function theorem (see Theorem \ref{G-ImplicitInfinite}) there exists  $\varepsilon >0$ such that the only solutions of the equation
$\nabla_u \Phi (u, \lambda) = 0$ are elements of ${\G}(\tilde{u}_0) \times \{\lambda\}$ for $\lambda \in (\lambda_0 - \varepsilon, \lambda_0 + \varepsilon).$  From this we obtain $\lambda_0 \not \in BIF.$ Therefore, if $\lambda_0 \in BIF,$
\begin{equation}\label{eq:warunek_dostateczny} \dim \ker \nabla^2_u \Phi (\tilde{u}_0, \lambda_0) > \dim ({\G}(\tilde{u}_0) \times \{\lambda_0\}).\end{equation}

Since $\G(\tilde{u}_0)=\Gamma(\tilde{u}_0),$ we conclude from \eqref{eq:orbita} and \eqref{eq:warunek_dostateczny} that $\ds \dim \ker \nabla^2_u \Phi (\tilde{u}_0, \lambda_0) > \dim ({\G}(\tilde{u}_0) \times \{\lambda_0\})=
\dim \ker \nabla^2 F (u_0),$ i.e.
$\dim \ker \left(Id - L_{\lambda_0 A}\right) > \dim \ker A.$
Using Corollary \ref{spektrumLA} we obtain that this condition is satisfied if and only if $\{(\alpha_j, \beta_k) \in \sigma(A) \times \sigma(-\Delta; B^N)\colon \beta_k = \lambda_0 \alpha_j\} \neq \{(0,0)\}.$
Therefore there are $(\alpha_j, \beta_k) \in \sigma(A)\setminus\{0\} \times \sigma(-\Delta,B^N)$ such that $\beta_k=\lambda_0\cdot \alpha_j$, i.e. $\lambda_0 \in \Lambda.$
\end{proof}

Fix $\lambda_0 \in \Lambda $ and choose $\varepsilon>0$  such that $\Lambda\cap[\lambda_0-\varepsilon,\lambda_0+\varepsilon]=\{\lambda_0\}$.
From the definition of $\Lambda$ such a choice is always possible.

Since $\lambda_0 \pm \varepsilon \notin \Lambda,$ Lemma \ref{fact:warunekkonieczny} implies that $\lambda_0\pm \varepsilon \notin BIF$ and therefore $\G(\tilde{u}_0) \subset \bH$ is an isolated critical orbit of the $\G$-invariant functionals $\Phi(\cdot, \lambda_0 \pm \varepsilon) \colon \bH \to \bR.$ From this and the properties of flows induced by gradient operators, we conclude that $\G(\tilde{u}_0)$ is also an isolated invariant set (in the sense of the equivariant Conley index theory, see \cite{Izydorek}) for the flows induced by the operators $-\nabla_u \Phi(\cdot, \lambda_0 \pm \varepsilon) $.
Therefore, the indices $CI_{{\G}}({\G}(\tilde{u}_0),-\nabla_u\Phi(\cdot, \lambda_0-\varepsilon))$, $CI_{{\G}}({\G}(\tilde{u}_0),-\nabla_u\Phi(\cdot, \lambda_0+\varepsilon))$ are well-defined. In the following we study when they are not equal.

Assume that $\sigma(\lambda_0 A)\cap \sigma(-\Delta; B^N)=\{\alpha_{j_1},\ldots, \alpha_{j_s}\}.$ We consider the conditions:
\begin{itemize}
\item[(C1)]\label{char1} $\lambda_0\neq 0$ and there is $i\in \{1,\ldots,s\}$ satisfying $\dim\bV_{-\Delta}(\alpha_{j_i})>1$,
\item[(C2)]\label{char2} $\lambda_0\neq 0$,  $\dim\bV_{-\Delta}(\alpha_{j_i})=1$ for every $i\in \{1,\ldots,s\}$  and $\dim \ker (Id-L_{\lambda_0A}) - \dim \ker A$ is an odd number,
\item[(C3)]\label{char3} $\lambda_0= 0$ and $\sum_{\alpha \in \sigma_+(A)} \mu_A(\alpha)-\sum_{\alpha \in \sigma_-(A)} \mu_A(\alpha)$ is odd.
\end{itemize}

\begin{Remark}
Note that we can reformulate conditions (C1)--(C3) in the following way:
\begin{enumerate}
\item[(C1')] $\lambda_0\neq 0$ and there is $i\in \{1,\ldots,s\}$ such that $\bV_{-\Delta}(\alpha_{j_i})$ is a nontrivial $SO(N)$-representation,
\item[(C2')] $\lambda_0\neq 0$,  $\dim\bV_{-\Delta}(\alpha_{j_i})=1$ for every $i\in \{1,\ldots,s\}$  and $\sum^s_{i=1} \mu_{\lambda_0 A}(\alpha_{j_i})-\mu_A(0)$ is odd,
\item[(C3')] $\lambda_0=0$ and $m-\dim \ker A$ is odd.
\end{enumerate}
Indeed,
\begin{enumerate}
\item $\dim\bV_{-\Delta}(\alpha_{j_i})>1$ if and only if $\bV_{-\Delta}(\alpha_{j_i})$ is a nontrivial $SO(N)$-representation, see Remark \ref{rem:nontriviality};
\item since $\dim\bV_{-\Delta}(\alpha_{j_i})=1$, from Corollary \ref{cor:kernel} we obtain $\dim \ker (Id-L_{\lambda_0A})=\sum^s_{i=1} \mu_{\lambda_0 A}(\alpha_{j_i})$;
\item since $\sum_{\alpha \in \sigma_+(A)} \mu_A(\alpha)+\sum_{\alpha \in \sigma_-(A)} \mu_A(\alpha)+\mu_A(0)=m$, if $m-\dim \ker A$ is odd, then so is $\sum_{\alpha \in \sigma_+(A)} \mu_A(\alpha)-\sum_{\alpha \in \sigma_-(A)} \mu_A(\alpha)$.
\end{enumerate}
\end{Remark}

\begin{Theorem}\label{th:zmiana_indeksu}
Assume that $\lambda_0 \in \Lambda $ and one of the conditions (C1)--(C3) is satisfied. Then $$\mathcal{CI}_{{\G}}({\G}(\tilde{u}_0),-\nabla_u\Phi(\cdot, \lambda_0-\varepsilon)) \neq \mathcal{CI}_{{\G}}({\G}(\tilde{u}_0),-\nabla_u\Phi(\cdot, \lambda_0+\varepsilon)).$$
\end{Theorem}

\begin{proof}
Denote by $\tilde{\bH}\subset\bH$ the linear subspace normal to ${\G}(\tilde{u}_0)$ at $\tilde{u}_0$, i.e. $\tilde{\bH}=T_{\tilde{u}_0}^{\perp} {\G}(\tilde{u}_0)\subset \bH$.
We start the proof with showing that we can reduce comparing the Conley indices $\cC\cI_{{\G}}({\G}(\tilde{u}_0),-\nabla_u\Phi(\cdot, \lambda_0\pm\varepsilon))$ to comparing Euler characteristics of some indices on the space $\tilde{\bH}$.

For $n\geq 1$ put $\bH^n=\bigoplus_{k=1}^{n}\bH_k$  and $\Phi^n=\Phi_{|\bH^n\times\bR}\colon \bH^n\times\bR \rightarrow \bR.$
Note that $\cG(\tilde{u}_0)=\Gamma(\tilde{u}_0) \subset T_{\tilde{u}_0} {\Gamma}(\tilde{u}_0) \oplus T_{\tilde{u}_0}^{\perp}{\Gamma}(\tilde{u}_0) \approx \bR^m \approx \bH_1.$ Therefore $\cG(\tilde{u}_0)$ is a critical orbit of $\Phi^n(\cdot, \lambda_0 \pm \varepsilon)$ for $n \geq 1$. Note that, from the choice of $\varepsilon$ and the definition of $\Phi^n$, it is a non-degenerate one.

Since $\nabla_u \Phi(\cdot,\lambda)$ is a completely continuous perturbation of the identity for all $\lambda\in\bR$, from the definition of the infinite dimensional equivariant Conley index, see \cite{Izydorek}, the assertion of the theorem is equivalent to
$$CI_{\G} (\G(\tilde{u}_0), -\nabla_u\Phi^n(\cdot, \lambda_0-\varepsilon)) \neq CI_{\G} (\G(\tilde{u}_0), -\nabla_u\Phi^n(\cdot, \lambda_0+\varepsilon))$$ for $n$ sufficiently large.

It is known that the $\cG$-action on $\bH$ given by \eqref{eq:action} defines a $\cG_{\tilde{u}_0}$-action on $\tilde{\bH}$. Recall that $\cG_{\tilde{u}_0}=\{e\} \times SO(N).$ Hence $\tilde{\bH}$ is an orthogonal $SO(N)$-representation.

For $n\geq 1$ put $\tilde{\bH}^n=\bH^n \cap \tilde{\bH}=T_{\tilde{u}_0}^{\perp} \Gamma (\tilde{u}_0)\oplus \bigoplus_{k=2}^{n}\bH_k $ and define $\Psi^n_{\pm}=\Phi^n(\cdot, \lambda_0 \pm \varepsilon)_{|\tilde{\bH}^n}\colon \tilde{\bH}^n \rightarrow \bR.$ From this definition the functionals $\Psi^n_{\pm}$ are $SO(N)$-invariant.
Since $\cG(\tilde{u}_0)$ is a non-degenerate critical orbit of  $\Phi^n(\cdot, \lambda_0\pm \varepsilon)$, $\tilde{u}_0\in \tilde{\bH}$ is a non-degenerate critical point  of $\Psi^n_{\pm}$. Hence $\{\tilde{u}_0\}$ is an isolated invariant set (in the sense of the Conley index theory) of the flows generated by $-\nabla \Psi^n _{\pm}$.

Note that  since $\cG_{\tilde{u}_0}=\{e\} \times SO(N)$, by Lemma \ref{admissible}  the pair $({\G}, {\G}_{\tilde{u}_0})$ is admissible. Therefore, using Fact \ref{cor:3.2}  we obtain that the assertion reduces to
$$\chi_{\cG_{\tilde{u}_0}}(CI_{{\G_{\tilde{u}_0}}}(\{\tilde{u}_0\},-\nabla\Psi^n_-)) \neq \chi_{\cG_{\tilde{u}_0}}( CI_{{\G_{\tilde{u}_0}}}(\{\tilde{u}_0\},-\nabla\Psi^n_+))$$
for $n\in \mathbb{N}$ sufficiently large. It is easy to see that this inequality is equivalent to
$$\chi_{SO(N)}(CI_{SO(N)}(\{\tilde{u}_0\},-\nabla\Psi^n_-)) \neq \chi_{SO(N)}( CI_{SO(N)}(\{\tilde{u}_0\},-\nabla\Psi^n_+)).$$
We proceed to show that there exists $n_0 \in \mathbb{N}$ such that for $n \geq n_0$
\begin{equation}\label{eq:krok2}
CI_{SO(N)}(\{\tilde{u}_0\},-\nabla \Psi^n_{\pm})=CI_{SO(N)}(\{\tilde{u}_0\},-\nabla \Psi^{n_0}_{\pm}).
\end{equation}

Let ${\nu} \in \mathbb{N}.$ For $\delta>0$ sufficiently small and $\lambda \in [\lambda_0 -\varepsilon, \lambda_0+\varepsilon]$ we define $SO(N)$-equivariant gradient homotopy
 $H_{\lambda}^{\nu}\colon (D_{\delta}(\tilde{u}_0,\tilde{\bH}^{\nu})\times [0,1], \partial D_{\delta}(\tilde{u}_0,\tilde{\bH}^{\nu}) \times [0,1]) \rightarrow (\tilde{\bH}^{\nu},\tilde{\bH}^{\nu}\setminus \{0\})$ by
$$H_{\lambda}^{\nu}(u,t)= u -\tilde{u}_0 - L_{\lambda A}  (u -\tilde{u}_0)+t \lambda_0 P_{\nu} \circ \nabla \eta (u - \tilde{u}_0),$$
where $P_{\nu}\colon \tilde{\bH} \rightarrow \tilde{\bH}^{\nu}$ is the orthogonal $SO(N)$-equivariant projection onto $\tilde{\bH}^{\nu}.$
Note that from Lemma \ref{operatorLlambdaA} we have $P_{\nu} \circ L_{\lambda A} = L_{\lambda A} \circ P_{\nu}$ and hence this homotopy is well-defined.

Let us denote by
$\xi_{\lambda}^{\nu}\colon\tilde{\bH}^{\nu}\to\bR$ the $SO(N)$-invariant potential of $H^{\nu}_{\lambda}(\cdot,0).$ It is clear that $\nabla \xi_{\lambda}^{\nu}\colon  \tilde{\h}^{\nu} \rightarrow \tilde{\h}^{\nu}$ is self-adjoint $SO(N)$-equivariant linear map and is given by the formula $\nabla \xi_{\lambda}^{\nu}=(Id-L_{\lambda A})_{|\tilde{\h}^{\nu}}.$ From the homotopy invariance of the Conley index, see Theorem \ref{thm:homotopy}, we obtain
\begin{equation}\label{eq:krok2cont}
    CI_{SO(N)}(\{\tilde{u}_0\},-\nabla\Psi^{\nu}_{\pm})= CI_{SO(N)}(\{\tilde{u}_0\},-\nabla\xi^{{\nu}}_{\lambda_0 \pm \varepsilon}).
\end{equation}

Recall that $(\beta_k)$ denotes the sequence of the eigenvalues of the Neumann Laplacian and note that  $\beta_k\to +\infty$. Therefore,  there exists $n_0\in\bN$ such that the inequalities $\frac{\beta_n-(\lambda_0\pm\varepsilon)\alpha_j}{1+\beta_n}>0$ hold for every $n\geq n_0$ and $\alpha_j \in\sigma(A)$.
Hence, by Corollary \ref{spektrumLA}, there exists $n_0\in \bN$ such that $m^-(\nabla\xi_{\lambda_0 \pm \varepsilon}^n)=m^-(\nabla\xi_{\lambda_0 \pm \varepsilon}^{n_0})$  for every $n\geq n_0$, where $m^-(\cdot)$ is the Morse index.
Since $(\nabla\xi^{n}_{\lambda_0 \pm \varepsilon})_{|\tilde{\bH}^{n_0}}=\nabla\xi^{n_0}_{\lambda_0 \pm \varepsilon}$, the eigenspaces corresponding to the negative eigenvalues of $\nabla\xi^{n}_{\lambda_0 \pm \varepsilon}$ and $\nabla\xi^{n_0}_{\lambda_0 \pm \varepsilon}$ are the same $SO(N)$-representations. Thus, from Theorem \ref{CIjakosfera},
\begin{equation*}
        CI_{SO(N)}(\{\tilde{u}_0\},-\nabla\xi^{n}_{\lambda_0 \pm \varepsilon})=    CI_{SO(N)}(\{\tilde{u}_0\},-\nabla\xi^{n_0}_{\lambda_0 \pm \varepsilon}),
\end{equation*}
which implies \eqref{eq:krok2}.

What is left is to show that
\begin{equation*}
\chi_{SO(N)}\left(CI_{SO(N)}(\{\tilde{u}_0\}, -\nabla \Psi^{n_0}_+)\right) \neq \chi_{SO(N)}\left(CI_{SO(N)}(\{\tilde{u}_0\}, -\nabla \Psi^{n_0}_-)\right).
\end{equation*}

Denote by $\cW(\lambda)$ the direct sum of the eigenspaces of $Id-L_{\lambda A}$ (i.e.  of $\nabla \xi^{n_0}_{\lambda}$) corresponding to the negative eigenvalues and by $\cV(\lambda)$ the eigenspace corresponding to the zero eigenvalue. Note that from Corollary \ref{spektrumLA},

$$\cW(\lambda)=
\Bigg(\bigoplus_{\alpha_j\in\sigma( A)} \ \bigoplus_{\substack{\beta_k \in \sigma(-\Delta;B^N)\\\beta_k<\lambda  \alpha_j}} \bV_{-\Delta}(\beta_k)^{\mu_{ A}(\alpha_{j})}\Bigg)\cap \tilde{\bH},$$
$$\cV(\lambda)=\Bigg(\bigoplus_{\alpha_j\in\sigma(A)} \ \bigoplus_{\substack{\beta_k \in \sigma(-\Delta;B^N)\\ \beta_k=\lambda  \alpha_j}} \bV_{-\Delta}(\beta_k)^{\mu_{ A}(\alpha_{j})}\Bigg)\cap \tilde{\bH}.$$

From Theorem \ref{CIjakosfera}, $CI_{SO(N)}(\{\tilde{u}_0\},-\nabla\xi^{n_0}_{\lambda_0 \pm \varepsilon})$ are $SO(N)$-homotopy types of  $S^{\cW(\lambda_0\pm\varepsilon)}.$
Hence, from \eqref{eq:krok2cont},
 \[ \chi_{SO(N)}\left(CI_{SO(N)}(\{\tilde{u}_0\},-\nabla\Psi^{n_0}_{\pm})\right)=
\chi_{SO(N)}\left(S^{\cW(\lambda_0\pm\varepsilon)}\right).\]

 \begin{enumerate}
\item Suppose that $\lambda_0>0$  and $\varepsilon$ is such that $\lambda_0 - \varepsilon>0$. Recall that $\beta_k \geq 0$ for all $\beta_k \in \sigma(-\Delta;B^N).$ Then
$\cW(\lambda_0+\varepsilon)=\cW(\lambda_0-\varepsilon)\oplus \cV(\lambda_0).$
If  the assumption (C1) is satisfied, then, by Theorem \ref{thm:nontrivialityofEC} and Remark  \ref{rem:nontriviality}, we obtain
$\chi_{SO(N)}(S^{\cV(\lambda_0)})\neq \bI\in U(SO(N)).$
Similarly, if (C2) is fulfilled, then $\cV(\lambda_0)$ is a trivial $SO(N)$-representation and, from Corollary \ref{spektrumLA} and the definition of $\tilde{\h}$, $\dim \cV(\lambda_0)=\dim \ker (Id-L_{\lambda_0A}) - \dim \ker A$ is odd. Therefore:
\begin{eqnarray*}
\chi_{SO(N)}(S^{\cV(\lambda_0)})
=(-1)^{\dim\cV(\lambda_0)}\chi_{SO(N)}\left(SO(N)/SO(N)^+\right)=-\bI.
\end{eqnarray*}
In both cases we have
\begin{eqnarray*}
\chi_{SO(N)}\left(CI_{SO(N)}(\{\tilde{u}_0\},-\nabla\Psi^{n_0}_{+})\right)= \chi_{SO(N)}(S^{\cW(\lambda_0-\varepsilon)\oplus \cV(\lambda_0)})=\\
= \chi_{SO(N)}(S^{\cW(\lambda_0-\varepsilon)})\star \chi_{SO(N)}(S^{\cV(\lambda_0)}) \neq \chi_{SO(N)}(S^{\cW(\lambda_0-\varepsilon)})=\\ = \chi_{SO(N)}\left(CI_{SO(N)}(\{\tilde{u}_0\},-\nabla\Psi^{n_0}_{-})\right).
\end{eqnarray*}
In the second equality we use the fact that $S^{\cW(\lambda_0-\varepsilon)\oplus \cV(\lambda_0)}$ is $SO(N)$-homeomorphic to $S^{\cW(\lambda_0-\varepsilon)}\wedge S^{\cV(\lambda_0)}$ and the formula for multiplication in $U(SO(N))$, see \eqref{eq:actionsUG}. Then we use invertibility of $\chi_{SO(N)}(S^{\cW(\lambda_0-\varepsilon)})$ in $U(SO(N))$,  see \cite{GolRyb1}.

\item Suppose that $\lambda_0<0$. Then
$\cW(\lambda_0-\varepsilon)=\cW(\lambda_0+\varepsilon)\oplus \cV(\lambda_0)
$
and hence
\begin{eqnarray*}
\chi_{SO(N)}\left(CI_{SO(N)}(\{\tilde{u}_0\},-\nabla\Psi^{n_0}_{+})\right)\neq \chi_{SO(N)}\left(CI_{SO(N)}(\{\tilde{u}_0\},-\nabla\Psi^{n_0}_{-})\right),
\end{eqnarray*}
as before.
\item Finally, suppose that $\lambda_0=0$. Then, since
\[
\cW(\pm\varepsilon)=
\bigoplus\limits_{\alpha_j\in\sigma_{\pm}(A)}\bV_{-\Delta}(0)^{\mu_{A}(\alpha_j)},
\]
and therefore $\cW(\pm\varepsilon)$ are trivial $SO(N)$-representations,
\begin{eqnarray*}
&&\chi_{SO(N)}\left(CI_{SO(N)}(\{\tilde{u}_0\},-\nabla\Psi^{n_0}_{\pm})\right)=\chi_{SO(N)}(S^{\cW(\pm\varepsilon)})=\\
&=&(-1)^{\dim\cW(\pm\varepsilon)}\cdot \chi_{SO(N)}\left(SO(N)/SO(N)^+\right)=(-1)^{\dim\cW(\pm\varepsilon)} \cdot\bI.
\end{eqnarray*}
Hence, because the assumption (C3) implies that  $\dim\cW(\varepsilon)-\dim\cW(-\varepsilon)$ is odd, we have
\[
\chi_{SO(N)}\left(CI_{SO(N)}(\{\tilde{u}_0\},-\nabla\Psi^{n_0}_{+})\right)\neq
\chi_{SO(N)}\left(CI_{SO(N)}(\{\tilde{u}_0\},-\nabla\Psi^{n_0}_{-})\right),
\]
which completes the proof.

\end{enumerate}

\end{proof}

Now we are in a position to prove one of the main results of our paper, namely the bifurcation theorems.

\begin{Theorem}\label{th:BIF}
Consider the system \eqref{eq:neumann} with the potential $F$ and $u_0 \in \nabla F^{-1}(0)$ satisfying assumptions (B1)--(B4). Assume that $\lambda_0 \in \Lambda $ and one of the conditions (C1)--(C3) is satisfied. Then a local bifurcation of solutions of \eqref{eq:neumann} occurs from the orbit ${\G}(\tilde{u}_0) \times \{\lambda_0\}$.
\end{Theorem}

\begin{proof}

From Theorem \ref{th:zmiana_indeksu} it follows that if one of the conditions (C1)--(C3) is satisfied  then
$\cC\cI_{{\G}}({\G}(\tilde{u}_0),-\nabla_u\Phi(\cdot, \lambda_{0}-\varepsilon)) \neq \cC\cI_{{\G}}({\G}(\tilde{u}_0),-\nabla_u\Phi(\cdot, \lambda_{0}+\varepsilon)),$ for sufficiently small $\varepsilon >0.$
Following for example the idea of the proof of Theorem 2.1 of \cite{SmoWass}, using the continuation property of the Conley index, one can prove that the change of the Conley index implies a local bifurcation of critical orbits.
\end{proof}

It is known that, in general, the change of the Conley index along the family of trivial solutions, does not imply the global bifurcation. However, using the relation between the Conley index and the degree for strongly indefinite functionals, under some assumptions one can prove the existence of connected sets of bifurcating solutions. It occurs that (C1)--(C3) are this kind of assumptions.

\begin{Theorem}\label{th:GLOB}
Consider the system \eqref{eq:neumann} with the potential $F$ and $u_0 \in \nabla F^{-1}(0)$ satisfying assumptions (B1)--(B4). Assume that $\lambda_0 \in \Lambda $ and one of the conditions (C1)--(C3) is satisfied. Then a global bifurcation of solutions of \eqref{eq:neumann} occurs from the orbit ${\G}(\tilde{u}_0) \times \{\lambda_0\}$.
\end{Theorem}

\begin{proof}

Let $\cU \subset \bH$ be an open, bounded and ${\G}$-invariant subset such that $\nabla_u \Phi (\cdot, \lambda_0\pm\varepsilon)^{-1}(0)\cap \cU={\G}(\tilde{u}_0).$ Denote by $\nabla_{\G}\textrm{-}\mathrm{deg}(\cdot, \cdot)$ the degree for equivariant gradient maps of the form completely continuous perturbation of the identity, defined in \cite{Ryb2005}.
From the definition of this degree, for $n_0$ sufficiently large,
\begin{eqnarray*}
\nabla_{{\G}}\textrm{-}\mathrm{deg}(\nabla_u\Phi(\cdot, \lambda_0\pm\varepsilon),\cU)
&=& \nabla_{{\G}}\textrm{-}\mathrm{deg}(\nabla_u\Phi^{n_0}(\cdot, \lambda_0\pm\varepsilon),\cU\cap \bH^{n_0})=\\&=&\chi_{{\G}}\left(CI_{{\G}}({\G}(\tilde{u}_0),-\nabla_u\Phi^{n_0}(\cdot, \lambda_0\pm\varepsilon))\right),
\end{eqnarray*}
where $\Phi^{n_0}$ is defined as in the proof of Theorem \ref{th:zmiana_indeksu}.
The latter equality is the relation between the Conley index and the degree proved by Gęba in \cite{Geba}, see also Corollary 1 in \cite{GolRyb}.

From Theorem \ref{th:zmiana_indeksu} and Fact \ref{cor:3.2} we have
$$\chi_{\G} (CI_{{\G}}({\G}(\tilde{u}_0),-\nabla_u\Phi^{n_0}(\cdot, \lambda_0-\varepsilon))) \neq \chi_{\G}(CI_{{\G}}({\G}(\tilde{u}_0),-\nabla_u\Phi^{n_0}(\cdot, \lambda_0+\varepsilon))).$$ Therefore $$\nabla_{{\G}}\textrm{-}\mathrm{deg}(\nabla_u\Phi(\cdot, \lambda_0-\varepsilon),\cU)\neq \nabla_{{\G}}\textrm{-}\mathrm{deg}(\nabla_u\Phi(\cdot, \lambda_0+\varepsilon),\cU).$$
From the equivariant version of the Rabinowitz alternative, see for example Theorem 3.3 of \cite{GolRyb1}, the change of the degree for ${\G}$-equivariant gradient maps implies a global bifurcation, so we obtain the assertion.
\end{proof}

In Theorem \ref{th:GLOB} we have proved that if the assumption (C3) is satisfied, then $0\in GLOB$. On the other hand, repeating the reasoning from the proof of this theorem it is easy to show that if  the number $\sum_{\alpha_j \in \sigma_+(A)} \mu_A(\alpha_j)-\sum_{\alpha_j \in \sigma_-(A)} \mu_A(\alpha_j)$ is even, then the Euler characteristics
$\chi_{\cG}\left(CI_{\cG}(\cG(\tilde{u}_0),-\nabla_u\Phi^{n_0}(\cdot,\varepsilon)\right)$ and $\chi_{\cG}\left(CI_{\cG}(\cG(\tilde{u}_0),-\nabla_u\Phi^{n_0}(\cdot,-\varepsilon))\right)$ are equal. Therefore, we do not know whether $0 \in GLOB$. However, under the assumption weaker than (C3)  we can prove the result concerning the local bifurcation.

\begin{Theorem}\label{thm:0} Consider the system \eqref{eq:neumann} with the potential $F$ and $u_0 \in \nabla F^{-1}(0)$ satisfying assumptions (B1)--(B4).  Assume that $\lambda_0 =0$ and $\sum_{\alpha_j \in \sigma_+(A)} \mu_A(\alpha_j)\neq\sum_{\alpha_j \in \sigma_-(A)} \mu_A(\alpha_j)$. Then a local bifurcation of solutions of \eqref{eq:neumann} occurs from the orbit ${\G}(\tilde{u}_0) \times \{0\}$.
\end{Theorem}
\begin{proof}
Using the notation of the proof of Theorem \ref{th:zmiana_indeksu}, we observe that $\cW(\pm\varepsilon)$ are trivial $SO(N)$-representations. Therefore $CI_{SO(N)}(\{\tilde{u}_0\},-\nabla\Psi^{n_0}_{\pm})$ are $SO(N)$-homotopy types of $S^{\dim \cW(\pm\varepsilon)}$.
Using information from \cite{PRS} (namely Theorem 3.1 and the equality (2.11)) and from \cite{Kawakubo} (Lemma 1.88) we obtain that $CI_{\cG}(\cG(\tilde{u}_0),-\nabla_u\Phi^{n_0}(\cdot,\pm \varepsilon))$ are $\cG$-homotopy types of
\[\left(\cG/\cG_{\tilde{u}_0}\times S^{\dim \cW(\pm\varepsilon)}\right)/\left(\cG/\cG_{\tilde{u}_0}\times \{*\}\right).\]
From Proposition 1.53 of \cite{Kawakubo}, we obtain that the above is $\cG$-homotopy equivalent to
\[X_{\pm}=\left(\cG(\tilde{u}_0)\times S^{\dim \cW(\pm\varepsilon)}\right)/\left(\cG(\tilde{u}_0)\times \{*\}\right).\]
But $X_+$ and $X_-$ are different $\cG$-homotopy types. Indeed, if $X_+$ and $X_-$ are the same $\cG$-homotopy types, then the orbit spaces $X_{+}/\cG$ and $X_{-}/\cG$ are the same homotopy types. This is impossible, since the spaces $X_{\pm}/\cG$ are homotopy types of $S^{\dim \cW(\pm\varepsilon)}$, see \cite{TomDieck}.
Analysis similar to that  in the proof of Theorem \ref{th:BIF} shows the assertion.
\end{proof}

%=================================================================================

\subsection{Symmetry breaking}\label{subsec:symmbreak}

In this section we consider the symmetry-breaking problem, i.e. the change of the isotropy groups of solutions of \eqref{eq:neumann} along connected sets. More precisely, we characterise bifurcation orbits of the equation \eqref{eq:neumann} at which the global symmetry-breaking phenomenon
occurs. Here and thereafter we use the notation of Section \ref{bifurcations}.
Recall that $\cT$ denotes the set of trivial solutions.

\begin{Definition}\label{def:symmbreak}
We say that a global symmetry-breaking phenomenon occurs at the orbit $\cG(\tilde{u}_0)\times\{\lambda_{0}\}$
if $\lambda_0\in GLOB$ and
 there exists $U\subset \bH\times\bR$ such that $\cG(\tilde{u}_0)\times\{\lambda_{0}\}\subset U$ and $\cG_{(u,\lambda)}\neq \cG_{(\tilde{u}_0, \lambda_0)}$ for all $(u,\lambda)\in (U\cap (\nabla_u\Phi)^{-1}(0))\setminus \cT$.
\end{Definition}

Note that since the group $\cG$ acts trivially on the set of parameters $\lambda$, the condition $\cG_{(u,\lambda)}\neq \cG_{(\tilde{u}_0, \lambda_0)}$ is equivalent to $\cG_u \neq \cG_{\tilde{u}_0}$.
In particular we are interested in studying $SO(N)$-symmetries of solutions. We say that the function $u$ satisfying $SO(N)_u=SO(N)$ is radially symmetric.

Our aim in this section is to prove the following characterisation of global symmetry-breaking phenomenon of solutions of \eqref{eq:neumann}:

\begin{Theorem}\label{thm:SymmBreak}
Consider the system \eqref{eq:neumann} with the potential $F$ and $u_0 \in \nabla F^{-1}(0)$ satisfying assumptions (B1)--(B4). Fix $\lambda_0\in\Lambda$ and suppose that $\sigma(\lambda_0 A)\cap \sigma(-\Delta; B^N)\setminus\{0\}=\{\alpha_{j_1},\ldots, \alpha_{j_s}\}$ and
 $\bV_{-\Delta}(\alpha_{j_i})^{SO(N)}=\{0\}$ for every $i=1,\ldots,s$.
 Then the global symmetry-breaking phenomenon occurs at the orbit $\cG(\tilde{u}_0)\times\{\lambda_{0}\}$.

\end{Theorem}
Note that the assumption  $\bV_{-\Delta}(\alpha_{j_i})^{SO(N)}=\{0\}$ means that there is no radially symmetric eigenfunction associated with $\alpha_{j_i}$.

To prove this theorem we first verify the following lemma:

\begin{Lemma}\label{lem:IsGr}
Fix $\lambda_0 \in \Lambda$.
 Then there exists $U\subset \bH\times\bR$ such that $\cG(\tilde{u}_0)\times\{\lambda_{0}\}\subset U$ and for all $(u,\lambda)\in (U\cap (\nabla_u\Phi)^{-1}(0))\setminus \cT$ there exists $\overline{u}\in\ker \nabla_u^2\Phi_{|\bH_1^{\perp}}(\tilde{u}_0,\lambda_0)\setminus \{0\}$ such that $\G_u\subset \G_{\overline{u}}$.
\end{Lemma}
\begin{proof}
Consider
$
\bU_1=\im \nabla_u^2\Phi_{|\bH_1^{\perp}}(\tilde{u}_0,\lambda_0)\oplus \bH_1$ and
$\bU_2=\ker \nabla_u^2\Phi_{|\bH_1^{\perp}}(\tilde{u}_0,\lambda_0).
$
Note that  $\bH=\bU_1\oplus\bU_2$ and the spaces $\bU_1$ and $\bU_2$ are ${\G}$-representations.
For $u\in\bH$ we put $u=(u_1,u_2)\in\bU_1\oplus\bU_2$.
In particular, since $\tilde{u}_0\in \bH_1$, we identify this element with $(\tilde{u}_0,0)\in \bU_1\oplus\bU_2$.

The equation
 \begin{equation}\label{SB}
\nabla_u\Phi(u,\lambda)=0
 \end{equation}
 is equivalent to the system
\begin{equation}\label{SB1}
\pi_1(\nabla_u\Phi(u_1,u_2,\lambda))=0,
\end{equation}
\begin{equation}\label{SB2}
\pi_2(\nabla_u\Phi(u_1,u_2,\lambda))=0.
\end{equation}
where $\pi_1\colon\bH\to\bU_1$ and $\pi_2\colon\bH\to\bU_2$ are ${\G}$-equivariant projections.
Moreover, since $\cG(\tilde{u}_0)\subset \bH_1\subset \bU_1$,
\[\dim\ker\nabla^2_u\Phi_{|\bU_1}(\tilde{u}_0,\lambda_{0})=\dim \cG(\tilde{u}_0),
\]
 i.e. $\cG (\tilde{u}_0)$ is a non-degenerate critical orbit of $\Phi(\cdot,\lambda_0)_{|\bU_1}$.
Therefore, by the equivariant implicit function theorem  (see Theorem \ref{G-ImplicitInfinite}) applied to the functional $\Phi\colon\bU_1\oplus(\bU_2 \times \bR) \rightarrow \bR$, the point $(0, \lambda_0)$ and the equation \eqref{SB1}, there exist open sets $\cO_{0}\subset \bU_2$, $\cO_{\lambda_0}\subset \bR$ such that $0\in \cO_{0},\lambda_0 \in \cO_{\lambda_0}$ and
a ${\G}$-equivariant map $\tau\colon \cG(\tilde{u}_0)\times \cO_{0}\times \cO_{\lambda_0}\to \bU_1 $ such that
\begin{enumerate}[(i)]
\item $\tau(u_1,0,\lambda_0)=u_1$ for $u_1\in\cG(\tilde{u}_0)$,
\item $\pi_1(\nabla_u\Phi(\tau(u_1,u_2,\lambda),u_2,\lambda))=0$ if $u_1 \in \cG(\tilde{u}_0), u_2\in \cO_{0}$ and $\lambda \in \cO_{\lambda_0}$ and these are the only solutions of $\pi_1(\nabla_u\Phi(u_1,u_2,\lambda))=0$ near the orbit if  $u_2\in \cO_{0}$ and $\lambda \in \cO_{\lambda_0}.$
\end{enumerate}

Hence all the solutions of the equation \eqref{SB1}, and consequently the solutions of \eqref{SB2} and  \eqref{SB}, can have (in the neighbourhood of the orbit) only the following isotropy groups:
\[
{\G}_{(\tau(u_1,u_2,\lambda),u_2,\lambda)}={\G}_{\tau(u_1,u_2,\lambda)}\cap {\G}_{u_2}\cap{\G}_{\lambda}=
{\G}_{\tau(u_1,u_2,\lambda)}\cap {\G}_{u_2}\subset {\G}_{u_2}.
\]

To finish the proof observe that in the case $u_2=0$ we have $(\tau(u_1,0,\lambda),0,\lambda)\in \bU_1\times\{0\}\times\bR$ for $u_1\in \cG(\tilde{u}_0)$, $\lambda\in\cO_{\lambda_0}$. Considering only the solutions of \eqref{SB1} and observing that such solutions in $\bU_1\times\{0\}\times\bR$ are the trivial ones, we obtain $(\tau(u_1,0,\lambda),0,\lambda)\in\cT$, which completes the proof.

\end{proof}

Lemma \ref{lem:IsGr} generalises the lemma due to Dancer from \cite{Dancer1979}. Dancer's result states that if the kernel of the second derivative of the functional at a bifurcation point does not contain nonzero radially-symmetric elements, then at a neighbourhood of this point all nontrivial solutions are not radial. This lemma cannot be applied to prove Theorem \ref{thm:SymmBreak} in the case $\dim \cG(\tilde{u}_0)>0$, since $\ker \nabla_u^2\Phi(\tilde{u}_0,\lambda_0)$ contains constant (and therefore radially symmetric) functions from the space tangent to the orbit.

\begin{proof}[Proof of Theorem \ref{thm:SymmBreak}]
Note that Theorem \ref{th:GLOB} implies that $\lambda_0\in GLOB$. Moreover, from Corollary \ref{cor:kernel} we have
\[
\ker \nabla_u^2\Phi_{|\bH_1^{\perp}}(\tilde{u}_0,\lambda_0)= \ker( Id - L_{\lambda_0 A})\cap \bH_1^{\perp}=\bV_{-\Delta}(\alpha_{j_1})^{\mu_{\lambda_0A}(\alpha_{j_1})}\oplus\ldots\oplus
\bV_{-\Delta}(\alpha_{j_s})^{\mu_{\lambda_0A}(\alpha_{j_s})}.
\]
Since $\alpha_{j_1},\ldots, \alpha_{j_s}\neq 0$ are such that $\bV_{-\Delta}(\alpha_{j_i})^{SO(N)}=\{0\}$ for every $i=1,\ldots,s$, we conclude that
\begin{equation}\label{eq:czymker}
\ker \nabla_u^2\Phi_{|\bH_1^{\perp}}(\tilde{u}_0,\lambda_0)^{SO(N)}=\{0\}.
\end{equation}
Lemma \ref{lem:IsGr} yields that there exists $U\subset \bH\times\bR$ such that if $\nabla_u \Phi(u,\lambda)=0$ and  $(u,\lambda)\in U\setminus \cT$ then there exists $\overline{u}\in\ker \nabla_u^2\Phi_{|\bH_1^{\perp}}(\tilde{u}_0,\lambda_0)\setminus \{0\}$ such that $\G_u\subset \G_{\overline{u}}$.
Since $\cG_{\tilde{u}_0}=\{e\}\times SO(N)$, to prove that $\cG_u\neq \cG_{\tilde{u}_0}$ it suffices to note that  the isotropy group of $\overline{u}$ is not of the form $H\times SO(N)$, where $H\in\sub (\Gamma)$.
Indeed, if ${\G}_{\overline{u}}= H\times SO(N)$, then $\overline{u}(\alpha^{-1} x)= \overline{u}(x)$ for every $\alpha\in SO(N)$, $x\in B^N$, i.e. $SO(N)_{\overline{u}}=SO(N)$ and therefore from \eqref{eq:czymker} we obtain $ \overline{u} =0$, which contradicts $\overline{u}\in\ker \nabla_u^2\Phi_{|\bH_1^{\perp}}(\tilde{u}_0,\lambda_0)\setminus \{0\}$.
\end{proof}

Note that if the assumptions of Theorem \ref{thm:SymmBreak} are satisfied, i.e. $\ker\nabla_u^2\Phi_{|\bH_1^{\perp}}(\tilde{u}_0,\lambda_0)^{SO(N)}=\{0\}$, then there is a neighbourhood $U$ of the bifurcation orbit such that all nontrivial solutions from $U$ are non-radial. In other words, in Theorem \ref{thm:SymmBreak} we obtain a connected family of orbits of non-radial solutions bifurcating from the set of radial ones.

\begin{Remark}\label{rem:radial}
Let $\lambda_0 \in BIF.$ By the proof of Lemma \ref{lem:IsGr} we deduce that there is a neighbourhood of the orbit $\cG(\tilde{u}_0)\times\{\lambda_0\}$ such that all nontrivial solutions of $\nabla_u\Phi(u,\lambda)=0$ can have only the isotropy groups of the form ${\G}_{\tau(u_1,u_2,\lambda)}\cap {\G}_{u_2}$. Note that $u_1\in \cG(\tilde{u}_0)$ and hence $\G_{u_1}=\{e\}\times SO(N)$.

Consider the additional assumption: $$\ker \nabla_u^2\Phi_{|\bH_1^{\perp}}(\tilde{u}_0,\lambda_0)^{SO(N)}=\ker \nabla_u^2\Phi_{|\bH_1^{\perp}}(\tilde{u}_0,\lambda_0).$$
Then ${\G}_{u_2}= \Gamma_{u_2}\times SO(N)$.
Therefore by the proof of Lemma \ref{lem:IsGr}, and since a ${\G}$-equivariant function $\tau$ increases isotropy groups (i.e. ${\G}_{(u_1,u_2,\lambda)} \subset {\G}_{\tau(u_1,u_2,\lambda)}$), we have
\[\G_{u_1}\cap\G_{u_2}= (\{e\}\times SO(N))\cap (\Gamma_{u_2}\times SO(N))=\{e\}\times SO(N)\subset {\G}_{\tau(u_1,u_2,\lambda)}\cap {\G}_{u_2},
 \]
 i.e. solutions of $\nabla_u\Phi(u,\lambda)=0$ in the neighbourhood of the orbit $ \cG(\tilde{u}_0) \times \{\lambda_0\}$ have isotropy groups of the form $H\times SO(N)$, where $H\in\sub(\Gamma)$. Hence all solutions from the neighbourhood of the orbit are radial.
\end{Remark}

\begin{Remark}
Fix $\lambda_0\in\Lambda$ and suppose that $\sigma(\lambda_0 A)\cap \sigma(-\Delta; B^N)\setminus\{0\}=\{\alpha_{j_1},\ldots, \alpha_{j_s}\}$ are such that $\alpha_{j_1},\ldots, \alpha_{j_s}\notin \cA_0$,
where $\cA_0$ is defined in Section \ref{sec:eigenspaces}. Then from Remark \ref{rem:bezA0} it follows that $\bV_{-\Delta}(\alpha_{j_i})^{SO(N)}=\{0\}$ and therefore the assumptions of Theorem \ref{thm:SymmBreak} are satisfied. Hence the global symmetry-breaking phenomenon occurs at the orbit $\cG(\tilde{u}_0)\times\{\lambda_{0}\}$.

\end{Remark}

%====================================================

\section{Illustration}\label{sec:illustration}

In this section  we discuss a few examples in order to illustrate the abstract results proved in the previous section. Using the properties of the eigenspaces of the Laplace operator (with the Neumann boundary conditions) on the ball, we verify assumptions (C1)--(C3). More precisely we apply the information collected in Subsection \ref{sec:eigenspaces}.

\medskip

\nt \bf{Example 1.} \rm
Consider the system \eqref{eq:neumann} for $N=2$ with the potential $F$ and $u_0 \in \nabla F^{-1}(0)$ satisfying assumptions (B1)--(B4). Assume that $\lambda_0 \in \bR \setminus\{0 \}$ and $\sigma(\lambda_0 A) \cap \sigma(-\Delta; B^2)\setminus\{0\}= \{\alpha\},$ where $\sqrt{\alpha}$ is not a root of $J_0'(x)=0$ for $J_0$ being the Bessel function of order $0$. Following the notation of Section \ref{sec:eigenspaces} it means that $\alpha \not \in \cA_0$.

In this situation, from Theorem \ref{thm:nieprzywiedlnosc} and Fact \ref{fact:opis} the assumption (C1) of Section 3 is satisfied. By Theorem \ref{th:GLOB}  we obtain that a global bifurcation  occurs  from the orbit $\cG(\tilde{u}_0) \times \{\lambda_0\}$.

Moreover, from Remark \ref{rem:bezA0} it follows that  $\bV_{-\Delta}(\alpha)^{SO(2)}=\{0\}.$ Then by Theorem \ref{thm:SymmBreak} the global symmetry breaking occurs at the orbit $\cG(\tilde{u}_0)\times \{\lambda_0\}.$

\medskip

\nt \bf{Example 2.} \rm
Consider the system \eqref{eq:neumann} for $N=2$ with the potential $F$ and $u_0 \in \nabla F^{-1}(0)$ satisfying assumptions (B1)--(B4).
Assume that $\lambda_0 \in \bR \setminus\{0 \},$ $\sigma(\lambda_0 A) \cap \sigma(-\Delta; B^2)\setminus\{0\}= \{\alpha_1, \ldots,\alpha_s \}$ and there exists $i\in \{1, \ldots,s\}$ such that $\sqrt{\alpha_i}$
is not a root of $J_0'(x)=0$.

As in Example 1, a global bifurcation  occurs  from the orbit $\cG(\tilde{u}_0) \times \{\lambda_0\}.$ If moreover   $\alpha_i \not \in \cA_0$ for all $i\in \{1, \ldots,s\}$,
then the global symmetry breaking occurs at the orbit $\cG(\tilde{u}_0)\times \{\lambda_0\}.$

\medskip

\nt \bf{Example 3.} \rm
Consider the system \eqref{eq:neumann} for $N=3$ with the potential $F$ and $u_0 \in \nabla F^{-1}(0)$ satisfying assumptions (B1)--(B4).
Assume that $\lambda_0 \in \bR \setminus\{0 \}$, $\sigma(\lambda_0 A) \cap \sigma(-\Delta; B^3)\setminus\{0\}= \{\alpha_1, \ldots,\alpha_s \}$ and  there exists $i\in \{1, \ldots,s\}$ such that
$\sqrt{\alpha_i}$ is not a solution of the equation:
$$J_{\frac12}'(x)-\frac{1}{2x}J_{\frac12}(x)=0,$$
 where $J_{\frac12}$ is the Bessel function of order $\frac12$.
Therefore $\alpha_i \not \in \cA_0.$

In this situation, since $\cH_l^3 \subset \bV_{-\Delta}(\alpha_i)$ for some $l > 0$ (by Fact \ref{fact:opis}),  the assumption (C1) is satisfied and from Theorem \ref{th:GLOB} we obtain that a global bifurcation  occurs  from the orbit $\cG(\tilde{u}_0) \times \{\lambda_0\}$.

Moreover, if $\alpha_i \not \in \cA_0$ for all $i\in \{1, \ldots,s\}$,
then  from Remark \ref{rem:bezA0} we conclude that  $\bV_{-\Delta}(\alpha_i)^{SO(3)}=\{0\}$ for all $i\in \{1, \ldots,s\}$. Therefore, by Theorem \ref{thm:SymmBreak} it follows that the global symmetry breaking occurs at the orbit $\cG(\tilde{u}_0)\times \{\lambda_0\}.$

\medskip

\nt \bf{Example 4.} \rm
Consider the system \eqref{eq:neumann} with the potential $F$ and $u_0 \in \nabla F^{-1}(0)$ satisfying assumptions (B1)--(B4).
Assume that $\lambda_0 \in \bR \setminus\{0 \}$ and that $\sigma(\lambda_0 A) \cap \sigma(-\Delta; B^N)\setminus\{0\}= \{\alpha_1, \ldots,\alpha_s \},$ where $\sqrt{\alpha_i}$ is a solution of the equation
\begin{equation*}
J'_{\frac{N-2}{2}}(x)-\frac{N-2}{2x} J_{\frac{N-2}{2}}(x)=0
\end{equation*}
for every $i \in \{1, \ldots, s\}$.

If there exists $i \in  \{1, \ldots,s\}$ such that  $\dim\bV_{-\Delta}(\alpha_i)>1$ then the assumption (C1) is satisfied and by Theorem \ref{th:GLOB} we obtain that a global bifurcation  occurs  from the orbit $\cG(\tilde{u}_0) \times \{\lambda_0\}$.

If $\dim\bV_{-\Delta}(\alpha_{i})=1$ for all $i\in \{1, \ldots,s\}$, then
we assume additionally that
$\sum_{i=1}^s \mu_{\lambda_0 A}(\alpha_i)-\mu_A(0)$ is an odd number. In this situation the assumption (C2) is satisfied and by Theorem \ref{th:GLOB}  we obtain that a global bifurcation  occurs  from the orbit $\cG(\tilde{u}_0) \times \{\lambda_0\}$.

Note that, if $\dim\bV_{-\Delta}(\beta_{i})=1$ for all $i\in \{1, \ldots,s\}$, then $\ker( Id - L_{\lambda_0 A})^{SO(N)}=\ker( Id - L_{\lambda_0 A})$ (see Remark \ref{rem:nontriviality}(2)).
Therefore, from Remark \ref{rem:radial}, we conclude that all nontrivial solutions at a neighbourhood of $\cG(\tilde{u}_0) \times \{\lambda_0\}$ (bifurcating from this orbit) are radial, i.e. there is no symmetry breaking at the orbit.

\medskip

\nt \bf{Example 5.} \rm Consider the system \eqref{eq:neumann} with the potential $F$ and $u_0 \in \nabla F^{-1}(0)$ satisfying assumptions (B1)--(B4).
Assume that $\lambda_0=0$.

If $m-\dim \ker A$ is odd, then the assumption (C3) is satisfied and we obtain a global bifurcation from the orbit $\cG(\tilde{u}_0) \times \{0\}.$ If $m-\dim \ker A > 0$, then Theorem \ref{thm:0} implies a local bifurcation from the orbit $\cG(\tilde{u}_0) \times \{0\}.$

As in Example 4, it is easy to see that all nontrivial solutions at a neighbourhood of the orbit are radial.

\section{Appendix}\label{sec:app}

In the following section, to make the paper self-contained, we collect some classical definitions and facts which we use to prove our main results.
\subsection{The equivariant  implicit function theorem}

Below we reformulate an equi\-va\-riant version of the implicit function theorem in infinite dimensional spaces, due to Dancer (see \cite{[Dancer]}, paragraph 3).

\begin{Theorem}\label{G-ImplicitInfinite}
Let $G$ be a compact Lie group and suppose that
\begin{enumerate}[(i)]
\item $\bH_1$, $\bH_2$ are Hilbert spaces, which are orthogonal $G$-representations,
\item $\Phi\colon \bH_1\oplus\bH_2\to \bR$ is a $G$-invariant functional of class $C^2$,
\item there is $v_0\in \bH_2$ such that
 $\nabla^2_u\Phi(u,v_0)$ is  Fredholm for every $u\in\bH_1$,
 there is $u_0\in\bH_1$ such that $\nabla_u\Phi(u_0,v_0)=0$ and $G(u_0)$  is a non-degenerate critical orbit of $\Phi(\cdot, v_0)$.
\end{enumerate}
Then there exist $\delta>0$ and a continuous $G$-invariant map $\tau\colon G(u_0)\times B_{\delta}(v_0,\bH_2)\to \bH_1$ such that
\begin{enumerate}
\item $\tau(u,v_0)=u$ on $G(u_0)$,
\item $\nabla_u\Phi(\tau(u,v),v)=0$ if $u\in G(u_0)$ and $v\in  B_{\delta}(v_0,\bH_2)$ and these are the only solutions of $\nabla_u\Phi(u,v)=0$ near $G(u_0)$ if $v\in  B_{\delta}(v_0,\bH_2)$,
\item for each $v\in  B_{\delta}(v_0,\bH_2)$, the map $u\mapsto \tau(u,v)$ is one-to-one.
\end{enumerate}

\end{Theorem}

\subsection{Equivariant Conley index}\label{subsec:CI}

In this subsection we collect properties of the equivariant Conley index. For a fuller treatment we refer to \cite{[Bartsch]}, \cite{Geba} in the finite dimensional case and to \cite{Izydorek} for the infinite dimensional case.

Let $G$ be a compact Lie group and suppose that $\Omega$ is a $G$-invariant subset of a finite dimensional $G$-representation $\bV$.
The $G$-equivariant Conley index of an isolated invariant set of a (local) flow is defined as a $G$-homotopy type of a pointed $G$-space, see \cite{[Bartsch]}, \cite{Geba}.
If $f\colon\Omega\to\bV$ is a $G$-equivariant map of class $C^1$, then it generates a local $G$-flow $\eta$, such that $\eta(x_0,\cdot)$ is the local solution of the problem $y'(t)=f(y(t))$, $y(0)=x_0$.
 We denote by $CI_G(S, f)$ the Conley index of an isolated invariant set $S$ of the flow generated by $f$.

Put $S^{\bV}=D_1(0,\bV)/\partial D_1(0,\bV)$ and denote by $[S^{\bV}]_G$ a $G$-homotopy type of a pointed $G$-space $S^{\bV}$. From the definition of the Conley index and the Hartman--Grobman theorem there follows (see also \cite{SmoWass}):

\begin{Theorem}\label{CIjakosfera}
Let $f\colon \bV\to\bR$ be a $G$-invariant map of class $C^2$ and suppose that $v_0 \in\bV$ is such that $G(v_0)=\{v_0\}$, $\nabla f(v_0)=0$ and $\det\nabla^2 f(v_0)\neq 0$. Then $CI_G(\{v_0\}, -\nabla f)=[S^{\bV^-}]_G$, where $\bV^-$ is the direct sum of eigenspaces of $\nabla^2 f(v_0)$ corresponding to the negative eigenvalues.
\end{Theorem}

The following theorem is a direct consequence of the Continuation Property of the Conley index, see \cite{[Bartsch]}:

\begin{Theorem}\label{thm:homotopy}(Homotopy invariance)
Let $v_0 \in \bV$ be such that $G(v_0)=\{v_0\}$ and suppose that $f\in C^2(\bV\times [0,1], \bR)$
is $G$-invariant. If $\nabla_v f(v_0, t)=0$ and $\det\nabla_v^2 f(v_0,t)\neq 0$
for every $t \in [0,1]$, then
$$CI_{G}(\{v_0\},\nabla_v f(\cdot,0))= CI_G(\{v_0\},\nabla_v f(\cdot,1)).$$
\end{Theorem}

The Conley index of a flow generated by a gradient map is homotopy type of a pointed finite $G$-CW-complex, see Proposition 5.6 of \cite{Geba} for the proof.
With a $G$-homotopy type of a pointed finite $G$-CW-complex $X$ one can associate a $G$-equivariant Euler characteristic $\chi_G(X)$, which is an element of the Euler ring $U(G)$ with the unit $\bI=\chi_G(G/G ^+)$.
 The actions in $U(G)$ are defined by
\begin{equation}\label{eq:actionsUG}
\left.\begin{array}{rcl}
\chi_G(X)+ \chi_G(Y)&=&\chi_G(X\vee Y),\\
\chi_G(X)\star \chi_G(Y)&=&\chi_G(X\wedge Y),
\end{array} \right.
\end{equation}
where $X\vee Y$ is the wedge sum and $X\wedge Y$ is the smash product of pointed finite $G$-CW-complexes $X,Y$.
The full description of this theory one can find for example in   \cite{TomDieck1}, \cite{TomDieck}.

The following theorem is an immediate consequence of Lemma 3.4 of \cite{[GaRyb]}:
\begin{Theorem}\label{thm:nontrivialityofEC}
If the group $G$ is connected and $\bV$ is a nontrivial $G$-representation, then $\chi_G(S^{\bV})\neq \bI\in U(G)$.
\end{Theorem}

Consider the potential $\varphi \colon \bR^n \times \bR \rightarrow \bR$ and assume that for $\lambda_-, \lambda_+ \in \bR$ the critical orbit $G(\tilde{u}_0)$  of $\varphi(\cdot, \lambda_{\pm})$ is non-degenerate. In Section  \ref{sec:main} we compare equivariant Conley indices $CI_{G} (G(\tilde{u}_0), \varphi(\cdot, \lambda_{\pm})).$ Using the result from \cite{PRS} one can reduce this problem to comparing the Euler characteristics of the Conley indices of potentials restricted to the space orthogonal to the orbit. More precisely, reasoning as in the proof of Corollary 3.2 of \cite{PRS}, from Theorem 3.1 of \cite{PRS} we obtain the following fact:

\begin{Fact}\label{cor:3.2}
Let $\Omega \subset \bR^n$ be an open and $G$-invariant subset and $\varphi \in C^2(\Omega\times \bR,\mathbb{R})$ be $G$-invariant. Moreover,  let $\lambda_-, \lambda_+ \in \bR$ and $G(\tilde{u}_0) \subset (\nabla_u \varphi(\cdot, \lambda_{\pm}))^{-1}(0)$ be a non-degenerate critical orbit.
Put $\phi=\varphi_{|T^{\perp}_{\tilde{u}_0}G(\tilde{u}_0)}$. If the pair $(G, G_{\tilde{u}_0})$ is admissible  and $$\chi_{G_{\tilde{u}_0}}(CI_{G_{\tilde{u}_0}}(\{ \tilde{u}_0\}, -\nabla_u\phi(\cdot, \lambda_-)) \neq \chi_{G_{\tilde{u}_0}}(CI_{G_{\tilde{u}_0}}(\{ \tilde{u}_0\}, -\nabla_u\phi(\cdot, \lambda_+))$$ then
$$CI_{G}(G(\tilde{u}_0), -\nabla_u\varphi(\cdot, \lambda_-)) \neq CI_{G}(G(\tilde{u}_0),- \nabla_u\varphi(\cdot, \lambda_+)).$$
Moreover,
$$\chi_{G}(CI_{G}(G(\tilde{u}_0), -\nabla_u\varphi(\cdot, \lambda_-))) \neq \chi_{G}(CI_{G}(G(\tilde{u}_0),- \nabla_u\varphi(\cdot, \lambda_+))).$$
\end{Fact}

Suppose now that $\cU$ is a $G$-invariant subset of an infinite dimensional Hilbert space, which is an orthogonal $G$-representation $\bH$.
The $G$-equivariant Conley index of an isolated invariant set of a (local) $G$-$\mathcal{LS}$-flow is defined as a $G$-homotopy type of a $G$-equivariant spectrum, see \cite{Izydorek}.
As before, if $F\colon\cU\to\bH$ is a $G$-equivariant map of class $C^1$ and it is a completely continuous perturbation of the identity, then it generates a local $G$-$\mathcal{LS}$-flow.
We denote by $\mathcal{CI}_G(S, F)$ the Conley index of an isolated invariant set $S$ of the flow generated by $F$.

\subsection{Eigenspaces of the Laplace operator}\label{sec:eigenspaces}

In this subsection we introduce basic pro\-per\-ties of the eigenspaces of the Laplace operator (with the Neumann boundary conditions) on the ball. More precisely, we study the problem:
\begin{equation} \label{eq:appeig}
\left\{
\begin{array}{rclcl}  -  \triangle u & = & \beta u  & \text{ in } & B^N \\
                   \frac{\partial u}{\partial \nu} & = &  0 &\text{ on    } & S^{N-1}.
\end{array}
\right.
\end{equation}
These properties are known, but it is difficult to find a reference in the literature, except for the case $N=2, 3$, see for example \cite{[MICH]}, \cite{Mucha}. To make our article self-contained, we sketch here the general case.

Let $\cH^N_l$ denote the linear space of harmonic, homogeneous polynomials of $N$ independent variables, of degree $l$, restricted to the sphere $S^{N-1}.$

\begin{Theorem}\label{thm:nieprzywiedlnosc}
The spaces $\cH^N_l$ are irreducible representations of the group $SO(N)$. Furthermore, if $l \geq 1$ then $\cH^N_l$ is a nontrivial representation of $SO(N)$ and for $l=0$ it is a trivial one. Moreover,
\[\dim\cH^N_l=\left\{\begin{array}{ccc}
1& \mathrm{ if } & N=2, l=0 \\
2 & \mathrm{if} & N=2, l \geq 1 \\
 (2l+N-2)\frac{(N-3+l)!}{l!(N-2)!} & \mathrm{if} & N \geq 3, l \geq 0.\end{array}  \right.\]
\end{Theorem}

For the proof of the irreducibility of the spaces $\cH^N_l$ we refer the reader to \cite{Gurarie} (Theorem 5.1). The proof of the latter part of the theorem can be found in \cite{Shimakura} (Theorem 4.1).

To find eigenspaces of the equation \eqref{eq:appeig} we write the Laplacian in polar coordinates $r\geq 0$, $\varphi=(\varphi_1,\ldots, \varphi_N)$, $0\leq \varphi_i< \pi$ for $i=1,\ldots, N-1$, $0\leq \varphi_N< 2\pi$:
\begin{equation*}
\Delta u = r^{1-N}\frac{\partial}{\partial r}\left(r^{N-1}\frac{\partial u}{\partial r}\right) + \frac{1}{r^2}\Delta_{S^{N-1}} u,
\end{equation*}
where $\Delta_{S^{N-1}}$ is the Laplace--Beltrami operator on $S^{N-1}$. Applying a standard separation  of variables $u(\varphi,r)=v(\varphi)\cdot f(r)$ to \eqref{eq:appeig},  we obtain the system
\begin{eqnarray}\label{eq:evsphere}
-\Delta_{S^{N-1}}v(\varphi)&=&\mu v(\varphi)\ \text{ on }\ S^{N-1},\\
\label{eq:radialpart}
r^2 f''(r)+(N-1)rf'(r)+\left(\beta r^2-\mu\right) f(r)&=&0\ \text{ on }\  (0,1), \\
\label{eq:bound1}
|f(0)|&<&\infty,\\
\label{eq:bound2}
f'(1)&=&0.
\end{eqnarray}
The equation \eqref{eq:evsphere} has solutions only if $\mu$ is an eigenvalue of $-\Delta_{S^{N-1}}$, i.e. $\mu=\mu_l: =l(l+N-2)$, $l=0,1,\ldots$, with associated eigenspaces equal  $\cH^N_l$, see \cite{Shimakura}. Substituting $\mu=\mu_l$, $\rho=\sqrt{\beta} r$ and $f(r)=g(\rho)/\rho^{\frac{N-2}{2}}$ into \eqref{eq:radialpart}, we get the Bessel equation of order $l+\frac{N-2}{2}$:
\begin{equation*}
\rho^2 g''(\rho)+\rho g'(\rho)+\left( \rho^2-\left(l+\frac{N-2}{2}\right)^2\right) g(\rho)=0\ \text{ on }\ (0,\sqrt{\beta}).
\end{equation*}
Using \eqref{eq:bound1} we obtain that the solution of this equation is $g(\rho)=C_lJ_{l+\frac{N-2}{2}}(\rho)$,  where $C_l\in\bR$ and $J_{l+\frac{N-2}{2}}$  is the Bessel function of the first kind of order $l+\frac{N-2}{2}$.

Since we are interested only in solutions satisfying \eqref{eq:bound2}, taking into consideration that
$f'(r)=\sqrt{\beta}(\sqrt{\beta}r)^{1-\frac{N}{2}}\left(g'(\sqrt{\beta}r)-\frac{N-2}{2\sqrt{\beta} r} g(\sqrt{\beta} r)\right)$,
we obtain that $\sqrt{\beta}$ satisfies the equation:
\begin{equation}\label{eq:beta}
J'_{l+\frac{N-2}{2}}(x)-\frac{N-2}{2x} J_{l+\frac{N-2}{2}}(x)=0.
\end{equation}
For $m\in \bN$ we denote by $x_{lm}$ the $m$th solution of \eqref{eq:beta} in $(0, \infty)$. Put $x_{00}=0$ and
$\cA_l=\{\beta_{lm} = x^2_{lm}\}_{m=1}^{\infty}$ for $l>0$  and $\cA_0=\{\beta_{0m}=x_{0m}^2\}_{m=0}^{\infty}$.

\begin{Fact}\label{fact:opis}
From the above considerations:
\begin{enumerate}
\item $\sigma(-\Delta;B^N)$ is the union of the sets $\cA_l$,
\item if $\beta\in \cA_l$, then $\cH^N_l\subset\bV_{-\Delta}(\beta)$, i.e. $\cH^N_l$ is $SO(N)$-equivalent to a subspace of $\bV_{-\Delta}(\beta)$.
 For $\beta\in\sigma(-\Delta;B^N)$
 we have $\bV_{-\Delta}(\beta)\approx_{SO(N)}\bigoplus\limits_{l\in\{l\geq0\colon \beta\in\cA_l\}} \cH^N_l$ (by $\approx_{SO(N)}$ we understand the equivalence relation of $SO(N)$-representations).
 \end{enumerate}
\end{Fact}

\begin{Remark}\label{rem:nontriviality}
Since from Theorem \ref{thm:nieprzywiedlnosc} we have $\dim\cH^N_0=1$ and  $\dim\cH^N_l>1$ for $l\geq 1$, it follows that for $\beta\in \sigma(-\Delta;B^N)$:
\begin{enumerate}
\item if $\dim\bV_{-\Delta}(\beta)>1$, then there exists $l>0$ such that $\cH^N_l\subset\bV_{-\Delta}(\beta)$ and thus $\bV_{-\Delta}(\beta)$ is a nontrivial $SO(N)$-representation,
\item if $\dim\bV_{-\Delta}(\beta)=1$, then $\bV_{-\Delta}(\beta)\approx_{SO(N)}\cH^N_0$ and therefore it is a trivial representation of $SO(N)$.
\end{enumerate}
\end{Remark}

\begin{Remark}\label{rem:bezA0}
From Theorem \ref{thm:nieprzywiedlnosc} we obtain that if  $\beta\in \sigma(-\Delta;B^N)$ and $\beta\notin\cA_0$, then $\bV_{-\Delta}(\beta)^{SO(N)}=\{0\}$.
\end{Remark}

To illustrate the above description of the eigenspaces, we will look more closely at the cases  $N=2,3$.

Suppose that $N=2$. Then, for $l\in\bN\cup\{0\}$, the equation \eqref{eq:beta} is of the form $J_l'(x)=0$ and therefore $x_{lm}$ is the $m$th solution of $J_l'(x)=0$ in $(0, \infty)$ and $x_{00}=0$.

\begin{Fact}\label{lem:eigenspace}
Under the above notation,
$\sigma(-\Delta; B^2 )=\bigcup_{l=0}^{\infty}\cA_l=\{\beta_{lm} = x^2_{lm}\}_{l=1,m=1}^{\infty} \cup \{\beta_{0m}=x_{0m}^2\}_{m=0}^{\infty}$ with corresponding eigenvectors  given by
\begin{enumerate}
\item $v_{lm}^1(r,\phi)=J_l(x_{lm}r)\cos l \varphi$ and $v_{lm}^2(r,\phi)=J_l(x_{lm}r)\sin l \varphi$ for  $\beta_{lm}$ in the case $l>0,$
 \item $v_{0m}(r,\phi)=J_0(x_{0m}r)$ for $\beta_{0m}$ in the case $l=0.$
\end{enumerate}
\end{Fact}

Note that from the above fact it follows that
$\cH^2_l\approx_{SO(2)}\operatorname{span}\{v_{lm}^1, v_{lm}^2\}$ for $l> 0$ and $\cH^2_0\approx_{SO(2)}\operatorname{span}\{ v_{0m}\}$.

\begin{Corollary} \label{cor:sone}
Let $\beta \in \sigma(-\Delta; B^2 ),$ then

\begin{enumerate}
\item If $\beta \in\cA_l$ for $l>0$, i.e. $\beta=\beta_{lm}$ for given $l,m>0$, then $\bV_{-\Delta}(\beta)$ is a nontrivial $\sone$-representation. Moreover,
if $\dim\bV_{-\Delta}(\beta)$ is even, then $\bV_{-\Delta}(\beta)^{SO(2)}=\{0\}$ and if $\dim\bV_{-\Delta}(\beta)$ is odd, then $\bV_{-\Delta}(\beta)^{SO(2)}\approx_{SO(2)}\cH^2_0$.
\item If $\beta \in\cA_0$, i.e. $\beta=\beta_{0m}$ for a given $m \in \bN$, then $\dim \bV_{-\Delta}(\beta)$ is an odd number. Moreover, if $\dim\bV_{-\Delta}(\beta)=1$, then $\bV_{-\Delta}(\beta)\approx_{SO(2)}\cH^2_0$ is a trivial $SO(2)$-representation.
\end{enumerate}
\end{Corollary}

Suppose now that $N=3$. Then, for $l\in\bN\cup\{0\}$, the equation \eqref{eq:beta} is of the form $J_{l+\frac12}'(x)-\frac{1}{2x}J_{l+\frac12}(x)=0$ and therefore $x_{lm}$ is the $m$th solution of this equation in $(0, \infty)$ and $x_{00}=0$.

\begin{Fact}\label{lem:eigenspaceN=3}
Under the above notation,
$\sigma(-\Delta; B^3 )=\bigcup_{l=0}^{\infty}\cA_l=\{\beta_{lm} = x^2_{lm}\}_{l=1,m=1}^{\infty} \cup \{\beta_{0m}=x_{0m}^2\}_{m=0}^{\infty}$ with corresponding eigenvectors:
\begin{enumerate}
\item for $\beta_{lm}$ in the case $l>0$:
\begin{eqnarray*}
v_{kml}^1(r,\varphi_1,\varphi_2)&=&\frac{1}{\sqrt{r}}J_{l+\frac12}( x_{lm}r) P_{lk}(\cos \varphi_1)\sin k\varphi_2, \\
 v_{kml}^2(r,\varphi_1,\varphi_2)&=&\frac{1}{\sqrt{r}}J_{l+\frac12}(x_{lm}r) P_{lk}(\cos \varphi_1)\cos k\varphi_2, \\
v_{0ml}(r,\varphi_1,\varphi_2)&=&J_{l+\frac12}(x_{lm}r)P_l(\cos \varphi_1),
\end{eqnarray*}
 where $k=1,\dots, l$ and $P_{lk}$, $P_l$ are Legendre functions,
\item for $\beta_{0m}$: $v_{0m0}(r,\varphi_1,\varphi_2)=J_{\frac12}(x_{0m}r)$.
\end{enumerate}
\end{Fact}

From the above fact it follows that $\cH^3_l\approx_{SO(3)}\operatorname{span}\{v_{0ml},v_{1ml}^1,v_{1ml}^2,\ldots v_{lml}^1, v_{lml}^2\}$ for $l>0$ and  $\cH^3_0\approx_{SO(3)}\operatorname{span}\{v_{0m0}\}$.

The description of $\cH^N_l$ in the general case can be found in \cite{Vilenkin} (Chapter IX).

%------------------------------------------------------------------------------------------

\end{document}